\documentclass[oneside,british]{amsart}
\usepackage[T1]{fontenc}
\usepackage[latin9]{inputenc}
\usepackage{amstext}
\usepackage{amsthm}
\usepackage{amssymb}

\makeatletter
\numberwithin{equation}{section}
\numberwithin{figure}{section}
\theoremstyle{plain}
\newtheorem{thm}{\protect\theoremname}
\theoremstyle{plain}
\newtheorem{conjecture}[thm]{\protect\conjecturename}
\theoremstyle{plain}
\newtheorem{cor}[thm]{\protect\corollaryname}
\theoremstyle{plain}
\newtheorem{lem}[thm]{\protect\lemmaname}
\theoremstyle{plain}
\newtheorem*{thm*}{\protect\theoremname}
\theoremstyle{remark}
\newtheorem*{claim*}{\protect\claimname}

\makeatother

\usepackage{babel}
\providecommand{\claimname}{Claim}
\providecommand{\conjecturename}{Conjecture}
\providecommand{\corollaryname}{Corollary}
\providecommand{\lemmaname}{Lemma}
\providecommand{\theoremname}{Theorem}

\begin{document}

\title{Proof of the exact overlaps conjecture for systems with algebraic
contractions}

\author{\noindent Ariel Rapaport}

\begin{abstract}
We establish the exact overlaps conjecture for iterated functions
systems on the real line with algebraic contractions and arbitrary
translations.
\end{abstract}

\maketitle

\section{\label{sec:Introduction}Introduction}

\subsection{\label{subsec:Background}Background}

Let $m\ge1$ and $\Phi=\{\varphi_{j}(x)=\lambda_{j}x+t_{j}\}_{j=0}^{m}$
be a finite set of contracting similarities of $\mathbb{R}$, so that
$0\ne\lambda_{j}\in(-1,1)$ and $t_{j}\in\mathbb{R}$ for each $0\le j\le m$.
Such a collection $\Phi$ is called a self-similar iterated function
system (IFS). It is well known that there exists a unique nonempty
compact $K\subset\mathbb{R}$, called the attractor of $\Phi$, which
satisfies the relation,
\begin{equation}
K=\cup_{j=0}^{m}\varphi_{j}(K)\:.\label{eq:def of K}
\end{equation}
The set $K$ is said to be self-similar.

Suppose additionally that $p=(p_{j})_{j=0}^{m}$ is a probability
vector with strictly positive coordinates. Then there exists a unique
Borel probability measure $\mu=\mu(\Phi,p)$ on $\mathbb{R}$ such
that,
\[
\mu=\sum_{j=0}^{m}p_{j}\cdot\varphi_{j}\mu,
\]
where $\varphi_{j}\mu$ is the push-forward of $\mu$ by $\varphi_{j}$.
Its support is equal to $K$, it is the unique stationary probability
measure for the random walk moving from $x\in\mathbb{R}$ to $\varphi_{j}(x)$
with probability $p_{j}$, and it is called the self-similar measure
corresponding to $\Phi$ and $p$.

The dimension theory of self-similar measures is a central area of
research in fractal geometry. It was proven by Feng and Hu \cite{FH}
that $\mu$ is always exact dimensional. That is, there exists a value
$\dim\mu\in[0,1]$, called the dimension of $\mu$, such that,
\[
\dim\mu=\underset{\delta\downarrow0}{\lim}\:\frac{\log\mu(x-\delta,x+\delta)}{\log\delta}\text{ for }\mu\text{-a.e. }x\in\mathbb{R}\:.
\]
As proven in \cite{FLR}, $\dim\mu$ agrees with the value given to
$\mu$ by other commonly used notions of dimension, such as the Hausdorff,
packing and entropy dimensions.

It turns out that in most cases $\dim\mu$ satisfies a certain formula
in terms of $p$ and the contractions vector $\lambda=(\lambda_{j})_{j=0}^{m}$.
Denote by $H(p)$ the entropy of $p$ and by $\chi$ the Lyapunov
exponents corresponding to $p$ and $\lambda$. That is,
\begin{equation}
H(p)=-\sum_{j=0}^{m}p_{j}\log p_{j}\text{ and }\chi=-\sum_{j=0}^{m}p_{j}\log|\lambda_{j}|,\label{eq:def of ent and lyap exp}
\end{equation}
where here and everywhere else in this paper the base of the $\log$
function is $2$. Set,
\begin{equation}
\beta=\beta(\Phi,p)=\min\{1,H(p)/\chi\},\label{eq:def of beta}
\end{equation}
then it is not hard to show that $\beta$ is always an upper bound
for $\dim\mu$ and that it is equal to $\dim\mu$ whenever the union
in (\ref{eq:def of K}) is disjoint. Moreover, it was proven by Jordan,
Pollicott and Simon \cite{JPS} that if $\lambda$ is kept fixed and
$|\lambda_{j}|\in(0,\frac{1}{2})$ for each $0\le j\le m$, then $\dim\mu=\beta$
for Lebesgue almost every selection of the translations $(t_{j})_{j=0}^{m}\in\mathbb{R}^{m+1}$.
A version of this result for sets was first established by Falconer
\cite{Fa}.

On the other hand there are cases in which it is obvious that dimension
drop occurs, i.e. that $\dim\mu$ is strictly less than $\beta$.
Denote the index set $\{0,...,m\}$ by $\Lambda$. For $n\ge1$ and
a word $j_{1}...j_{n}=w\in\Lambda^{n}$ set,
\begin{equation}
\varphi_{w}=\varphi_{j_{1}}\circ...\circ\varphi_{j_{n}}\text{ and }\lambda_{w}=\lambda_{j_{1}}\cdot...\cdot\lambda_{j_{n}}\:.\label{eq:def of compo}
\end{equation}
The IFS $\Phi$ is said to have exact overlaps if the semigroup generated
by its elements is not free. Since the members of $\Phi$ are contractions,
this is equivalent to the existence of $n\ge1$ and distinct words
$w_{1},w_{2}\in\Lambda^{n}$ with $\varphi_{w_{1}}=\varphi_{w_{2}}$.
It is not difficult to see that $\dim\mu<\beta$ whenever $\Phi$
has exact overlaps and $\dim\mu<1$. The following folklore conjecture
says that these are the only circumstances in which dimension drop
can occur. A version of it for sets was stated, probably for the first
time, by Simon \cite{Si}.
\begin{conjecture}
\label{conj:exact overlaps}Suppose that $\dim\mu<\beta$ then $\Phi$
has exact overlaps.
\end{conjecture}

A major step towards the verification of Conjecture \ref{conj:exact overlaps}
was achieved by Hochman \cite{Ho1}. For $n\ge1$ set,
\begin{equation}
\Delta_{n}=\min\left\{ \left|\varphi_{w_{1}}(0)-\varphi_{w_{2}}(0)\right|\::\:w_{1},w_{2}\in\Lambda^{n},\:w_{1}\ne w_{2}\text{ and }\lambda_{w_{1}}=\lambda_{w_{2}}\right\} \:.\label{eq:def of Delta}
\end{equation}
It always holds that $\Delta_{n}\overset{n}{\rightarrow}0$ at a rate
which is at least exponential, and that $\Delta_{n}=0$ for some $n\ge1$
if and only if $\Phi$ has exact overlaps. The main result in \cite{Ho1}
says that if $\dim\mu<\beta$ then $\Delta_{n}\overset{n}{\rightarrow}0$
super-exponentially, that is
\[
\underset{n}{\lim}\:\frac{1}{n}\log\Delta_{n}=-\infty\:.
\]
A version of this for $L^{q}$ dimensions was recently obtained by
Shmerkin \cite[Theorem 6.6]{Sh}.

Two applications of Hochman's result are especially relevant to the
present paper. It is not hard to see that if $\lambda_{0},...,\lambda_{m},t_{0},...,t_{m}$
are all algebraic numbers and $\Delta_{n}\overset{n}{\rightarrow}0$
super-exponentially, then in fact $\Phi$ must have exact overlaps.
Relaying on this observation, Conjecture \ref{conj:exact overlaps}
is established in \cite[Theorem 1.5]{Ho1} for the case of algebraic
parameters. The second application verifies a conjecture of Furstenberg
regarding projections of the one-dimensional Sierpinski gasket (see
e.g. \cite[Question 2.5]{PS}). Stated with the notation introduced
above, it is proven in \cite[Theorem 1.6]{Ho1} that Conjecture \ref{conj:exact overlaps}
is valid when $m=2$ and
\[
\lambda=p=(\frac{1}{3},\frac{1}{3},\frac{1}{3})\:.
\]

Another important step towards Conjecture \ref{conj:exact overlaps}
was recently achieved by Varj{\'u} \cite{Var}. He has shown that
if $\mu$ is a Bernoulli convolution, that is if in the notation above
\[
m=1,\:\lambda_{0}=\lambda_{1}>0,\:t_{0}=-1\text{ and }t_{1}=1,
\]
then $\dim\mu=\beta$ whenever $\lambda_{0}$ is transcendental. Together
with the result mentioned above regarding systems with algebraic parameters,
this verifies Conjecture \ref{conj:exact overlaps} for the family
of Bernoulli convolutions.

Given Hochman's and Shmerkin's results, it is natural to ask whether
$\Phi$ has exact overlaps whenever $\Delta_{n}\overset{n}{\rightarrow}0$
super-exponentially. Recently, examples have been constructed by Baker
\cite{Ba} and independently by Bárány and Käenmäki \cite{BK}, which
show that this is not necessarily true. In Baker's construction the
maps in the IFS all contract by a rational number, and so it is especially
relevant to the present paper. In a joint work with P. Varj{\'u}
\cite{RV} we will treat a family of self-similar measures which is
closer to the example from \cite{BK}.

\subsection{Results}

The following theorem is our main result. It verifies Conjecture \ref{conj:exact overlaps}
for the case of algebraic contractions and arbitrary translations.
\begin{thm}
\label{thm:main thm}Let $m\ge0$ and $\Phi=\{\varphi_{j}(x)=\lambda_{j}x+t_{j}\}_{j=0}^{m}$
be a self-similar IFS on $\mathbb{R}$. Suppose that $\lambda_{0},...,\lambda_{m}$
are all algebraic numbers and that $\Phi$ has no exact overlaps.
Let $p=(p_{j})_{j=0}^{m}$ be a probability vector and denote by $\mu$
the self-similar measure corresponding to $\Phi$ and $p$. Then $\dim\mu=\beta$,
where $\beta$ is as defined in (\ref{eq:def of beta}).
\end{thm}

A version for sets of the conjecture follows directly from the last
theorem in the case of algebraic contractions. Given an IFS $\Phi$
as above denote by $\dim_{s}\Phi$ its similarity dimension, that
is $\dim_{s}\Phi$ is the unique $s\ge0$ which satisfies the equation,
\[
\sum_{j=0}^{m}|\lambda_{j}|^{s}=1\:.
\]
It is not hard to see that $\min\{1,\dim_{s}\Phi\}$ is always an
upper bound for $\dim_{H}K$, where $K$ is the attractor of $\Phi$
and $\dim_{H}$ stands for Hausdorff dimension. Moreover, the equality
\begin{equation}
\dim_{H}K=\min\{1,\dim_{s}\Phi\}\label{eq:dim(K)=00003Dexpected}
\end{equation}
is satisfied when the union in (\ref{eq:def of K}) is disjoint or,
more generally, if $\Phi$ satisfies the so-called open set condition
(see for instance \cite[Chapter 2.1]{BP}). The version for sets of
Conjecture \ref{conj:exact overlaps} says that (\ref{eq:dim(K)=00003Dexpected})
holds whenever $\Phi$ has no exact overlaps.
\begin{cor}
\label{cor:for sets}Let $m\ge0$ and $\Phi=\{\varphi_{j}(x)=\lambda_{j}x+t_{j}\}_{j=0}^{m}$
be a self-similar IFS on $\mathbb{R}$. Suppose that $\lambda_{0},...,\lambda_{m}$
are all algebraic numbers and that $\Phi$ has no exact overlaps.
Let $K$ be the attractor of $\Phi$, then (\ref{eq:dim(K)=00003Dexpected})
is satisfied.
\end{cor}

\begin{proof}
Write $s=\dim_{s}\Phi$ and denote by $p$ the probability vector
$(|\lambda_{j}|^{s})_{j=0}^{m}$. Let $\beta$ be as defined in (\ref{eq:def of beta}),
then $\beta=\min\{1,s\}$. Let $\mu$ be the self-similar measure
corresponding to $\Phi$ and $p$. By Theorem \ref{thm:main thm}
we have $\dim\mu=\min\{1,s\}$. Additionally, since $\mu$ is supported
on $K$ it follows that $\dim_{H}K\ge\dim\mu$. This completes the
proof of the corollary.
\end{proof}
For our next application suppose that $\beta<1$. As mentioned in
Section \ref{subsec:Background}, if $\Phi$ has exact overlaps then
necessarily $\dim\mu<\beta$. If instead it only holds that $\Delta_{n}\overset{n}{\rightarrow}0$
super-exponentially, then in general it is not clear whether there
is dimension drop or not. In fact, until now there was no single example
in which $\dim\mu$ was analysed in such a situation. Combining our
result with Baker's construction we are able to obtain such an example.
We say that an IFS is homogeneous if all of its maps have the same
contraction part.
\begin{cor}
Let $(\epsilon_{n})_{n\ge1}$ be an arbitrary sequence of positive
real numbers. Then there exists a homogeneous self-similar IFS $\Phi=\{\varphi_{j}(x)=\lambda x+t_{j}\}_{j=0}^{m}$
on $\mathbb{R}$ such that,
\begin{enumerate}
\item $\lambda$ is a rational number;
\item $\Phi$ has no exact overlaps;
\item $\Delta_{n}\le\epsilon_{n}$ for all $n\ge1$, where $\Delta_{n}$
is as defined in (\ref{eq:def of Delta});
\item $\dim_{s}\Phi<1$, and in particular $\beta(\Phi,p)<1$ for every
probability vector $p=(p_{j})_{j=0}^{m}$;
\item $\dim\mu(\Phi,p)=\beta(\Phi,p)$ for every probability vector $p=(p_{j})_{j=0}^{m}$,
where recall that $\mu(\Phi,p)$ is the self-similar measure corresponding
to $\Phi$ and $p$;
\item $\dim_{H}K=\dim_{s}\Phi$, where $K$ is the attractor of $\Phi$.
\end{enumerate}
\end{cor}

\begin{proof}
The existence of a homogeneous IFS $\Phi$ which satisfies the first
four properties follows from \cite[Theorem 1.3 and Remark 2.2]{Ba}.
The last two properties follow from Theorem \ref{thm:main thm} and
Corollary \ref{cor:for sets}.
\end{proof}

\subsection{About the proof}

We derive Theorem \ref{thm:main thm} from the following more general
statement, which concerns also some self-similar measures in higher
dimensions. This enables us to prove it by backward induction on the
dimension of the ambient space.
\begin{thm}
\label{thm:main thm inductive}Let $m\ge2$ be an integer, $\lambda=(\lambda_{j})_{j=0}^{m}$
be algebraic numbers in $(-1,1)\setminus\{0\}$ and $p=(p_{j})_{j=0}^{m}$
be a probability vector with strictly positive coordinates. Then for
every $1\le d<m$ the following holds. Let $(t_{j})_{j=1}^{m}=t\in(\mathbb{R}^{d})^{m}$
be such that,
\begin{equation}
\mathrm{span}\{t_{1},...,t_{m}\}=\mathbb{R}^{d},\label{eq:assumptions on t}
\end{equation}
and let $\Phi_{t}=\{\varphi_{t,j}\}_{j=0}^{m}$ be the self-similar
IFS on $\mathbb{R}^{d}$ with,
\[
\varphi_{t,0}(x)=\lambda_{0}x\text{ and }\varphi_{t,j}(x)=\lambda_{j}x+t_{j}\text{ for each }x\in\mathbb{R}^{d}\text{ and }1\le j\le m\:.
\]
Denote by $\mu_{t}$ the self-similar measure on $\mathbb{R}^{d}$
which corresponds to $\Phi_{t}$ and $p$. Suppose that $\Phi_{t}$
has no exact overlaps, then
\begin{equation}
\dim\mu_{t}\ge\min\{1,H(p)/\chi\},\label{eq:conclusion of thm}
\end{equation}
where $H(p)$ and $\chi$ are as defined in (\ref{eq:def of ent and lyap exp}).
\end{thm}

\begin{proof}[Proof of Theorem \ref{thm:main thm} given Theorem \ref{thm:main thm inductive}]
When $m=0$ the theorem is trivial. Suppose that $m=1$, then we
may assume, by conjugating the maps in $\Phi$ by an appropriate invertible
affine map, that $t_{0}=0$ and $t_{1}=1$. Thus, the theorem in this
case follows from Hochman's result mentioned above regarding systems
with algebraic parameters. When $m\ge2$ the theorem follows directly
from Theorem \ref{thm:main thm inductive} with $d=1$, and by conjugating
the maps in the IFS $\Phi$ so that $t_{0}=0$.
\end{proof}
Let us give an informal sketch for the proof of Theorem \ref{thm:main thm inductive}.
Given $d\ge1$ and $t\in(\mathbb{R}^{d})^{m}$ let $\Phi_{t}=\{\varphi_{t,j}\}_{j=0}^{m}$
and $\mu_{t}$ be as in the statement of the theorem. Recall that
$\Lambda=\{0,...,m\}$, and for $n\ge1$ and $w_{1},w_{2}\in\Lambda^{n}$
set
\[
L_{w_{1},w_{2}}(t)=\varphi_{t,w_{1}}(0)-\varphi_{t,w_{2}}(0)\text{ for }t\in\mathbb{R}^{m},
\]
where $\varphi_{t,w_{1}}$ and $\varphi_{t,w_{2}}$ are as defined
in (\ref{eq:def of compo}). It is easy to see that $L_{w_{1},w_{2}}$
is a linear functional on $\mathbb{R}^{m}$. Moreover, if $\{f_{1},...,f_{m}\}$
is the dual basis of the standard basis of $\mathbb{R}^{m}$, then
\begin{equation}
L_{w_{1},w_{2}}=\sum_{j=1}^{m}P_{j}(\lambda)f_{j}\text{ for some }P_{0},...,P_{m}\in\mathcal{P}(1,n),\label{eq:coeff of L}
\end{equation}
where $\mathcal{P}(1,n)$ is the set of all $P\in\mathbb{Z}[X_{0},...,X_{m}]$
with $\deg P<n$ and coefficients $\pm1$ or $0$. Denote by $\mathcal{L}_{n}$
the collection of all $L_{w_{1},w_{2}}$ with $w_{1},w_{2}\in\Lambda^{n}$,
$w_{1}\ne w_{2}$ and $\lambda_{w_{1}}=\lambda_{w_{2}}$.

As mentioned above the proof is carried out by backward induction
on $d$. Thus, let $1\le d<m$ and assume that the theorem holds for
all $d<d'<m$. Suppose that
\[
((t_{j}^{l})_{l=1}^{d})_{j=1}^{m}=(t_{j})_{j=1}^{m}=t\in(\mathbb{R}^{d})^{m}
\]
is such that (\ref{eq:assumptions on t}) is satisfied but (\ref{eq:conclusion of thm})
does not hold. Our objective is to show that $\Phi_{t}$ has exact
overlaps.

Let $\delta>0$ be small with respect to all previous parameters.
The starting point of the argument is a theorem of Hochman from \cite{Ho2}
regarding self-similar measures in $\mathbb{R}^{d}$. By applying
that theorem we deduce that for all sufficiently large $n\ge1$ there
exists a collection $\mathcal{A}_{n}\subset\mathcal{L}_{n}$, which
is large in a sense to be made precise during the actual proof, such
that
\begin{equation}
|L((t_{j}^{l})_{j=1}^{m})|\le\delta^{n}\text{ for each }1\le l\le d\text{ and }L\in\mathcal{A}_{n}\:.\label{eq:L is small}
\end{equation}
Denote by $r_{n}$ the dimension of the linear span of the members
of $\mathcal{A}_{n}$. By (\ref{eq:coeff of L}) and (\ref{eq:L is small}),
by assuming that $\delta$ is sufficiently small with respect to $\lambda$
and $t$, and since $\lambda_{0},...,\lambda_{m}$ are algebraic,
it is not hard to show that $r_{n}\le m-d$.

Now there are two options to consider. First assume that $r_{n}=m-d$
for all sufficiently large $n$. In this case we are able to show
that $\Phi_{t}$ has exact overlaps by extending the argument from
\cite{Ho1} used in the proof of Furstenberg's conjecture mentioned
above. Denote the annihilator of $\mathcal{A}_{n}$ by $V_{n}$, that
is
\[
V_{n}=\{x\in\mathbb{R}^{m}\::\:L(x)=0\text{ for all }L\in\mathcal{A}_{n}\}\:.
\]
By (\ref{eq:L is small}) it follows that $V_{n}$ and $V_{n+1}$
are both $\delta^{n}$-close to
\[
U=\mathrm{span}\{(t_{j}^{l})_{j=1}^{m}\::\:1\le l\le d\},
\]
which implies that they are $\delta^{n}$-close to each other. From
this and the information on the coefficients of members of $\mathcal{L}_{n}$,
it follows that in fact we must have $V_{n}=V_{n+1}$. This gives
$U=V_{n}$, from which it follows that $L((t_{j}^{l})_{j=1}^{m})=0$
for all $1\le l\le d$ and $L\in\mathcal{A}_{n}$. By the definition
of $\mathcal{L}_{n}$ this implies that $\Phi_{t}$ has exact overlaps.

In the second case we assume that there exist $1\le r<m-d$ and an
increasing sequence $\{n_{k}\}_{k\ge1}$ such that $r_{n_{k}}=r$
for all $k\ge1$. Set $d'=m-r$ and note that $d<d'<m$. For every
$k\ge1$ we choose
\[
((s_{k,j}^{l})_{l=1}^{d'})_{j=1}^{m}=(s_{k,j})_{j=1}^{m}=s_{k}\in([-1,1]^{d'})^{m},
\]
such that $\{s_{k,j}\}_{j=1}^{m}$ contains the standard basis of
$\mathbb{R}^{d'}$ and,
\begin{equation}
L((s_{k,j}^{l})_{j=1}^{m})=0\text{ for all }1\le l\le d'\text{ and }L\in\mathcal{A}_{n_{k}}\:.\label{eq:sketch zero of functionals}
\end{equation}
By moving to a subsequence we may assume that there exists $s\in(\mathbb{R}^{d'})^{m}$
such that $s_{k}\overset{k}{\rightarrow}s$. Now by (\ref{eq:sketch zero of functionals}),
since the collections $\mathcal{A}_{n_{k}}$ are sufficiently large
and by the lower semi-continuity of the dimension of self-similar
measures, it follows that $\dim\mu_{s}<\beta$. From this and the
induction hypothesis we get that $\Phi_{s}$ has exact overlaps. In
the last part of the proof we show that this, together with (\ref{eq:coeff of L})
and (\ref{eq:L is small}), implies that $\Phi_{t}$ has exact overlaps,
which completes the proof of the theorem.

It is interesting to consider Baker's construction in the context
of the proof just described. Since in this example $\Delta_{n}\overset{n}{\rightarrow}0$
super-exponentially, where $\Delta_{n}$ is defined in (\ref{eq:def of Delta}),
the collections $\mathcal{A}_{n}$ are necessarily nonempty. On the
other hand, since there are no exact overlaps, the preceding argument
shows that $\underset{n}{\liminf}\:r_{n}<m-d$ (where $d=1$) and
that the collections $\mathcal{A}_{n}$ cannot be large in a manner
resulting in a dimension drop.

\subsection*{Structure of the paper}

In the next section we make some preparations for the proof of Theorem
\ref{thm:main thm inductive}. In Section \ref{sec:Proof-of-Theorem}
we carry out the proof.

\subsection*{Acknowledgment}

I am grateful to P. Varj{\'u} for many inspiring discussions conducted
while working on the paper \cite{RV}, and for his comments on a previous
version of this paper. I would also like to thank M. Hochman and P.
Shmerkin for helpful remarks. This research was supported by the Herchel
Smith Fund at the University of Cambridge.

\section{Preparations for the proof of Theorem \ref{thm:main thm inductive}}

\subsection{\label{subsec:Some-notations}Some notations}

For the rest of this paper fix an integer $m\ge2$, a probability
vector $p=(p_{j})_{j=0}^{m}$ with strictly positive coordinates and
a vector $\lambda=(\lambda_{j})_{j=0}^{m}$ such that $\lambda_{j}$
is an algebraic number in $(-1,1)\setminus\{0\}$ for each $0\le j\le m$.
Recall that the base of the $\log$ function is always $2$ and set,
\[
H(p)=-\sum_{j=0}^{m}p_{j}\log p_{j},\:\chi=-\sum_{j=0}^{m}p_{j}\log|\lambda_{j}|\text{ and }\beta=\min\{1,H(p)/\chi\}\:.
\]

For an integer $d\ge1$ we shall often write $\mathbb{R}^{dm}$ in
place of $(\mathbb{R}^{d})^{m}$ when there is no risk of confusion.
Given $(t_{j})_{j=1}^{m}=t\in\mathbb{R}^{dm}$ let $\Phi_{t}=\{\varphi_{t,j}\}_{j=0}^{m}$
be the self-similar IFS on $\mathbb{R}^{d}$ with,
\[
\varphi_{t,0}(x)=\lambda_{0}x\text{ and }\varphi_{t,j}(x)=\lambda_{j}x+t_{j}\text{ for all }x\in\mathbb{R}^{d}\text{ and }1\le j\le m\:.
\]
Denote by $K_{t}$ the attractor of $\Phi_{t}$, that is $K_{t}$
is the unique nonempty compact subset of $\mathbb{R}^{d}$ with,
\[
K_{t}=\cup_{j=0}^{m}\varphi_{t,j}(K_{t})\:.
\]
Write $\mu_{t}$ for the self-similar measure corresponding to $\Phi_{t}$
and $p$, i.e. $\mu_{t}$ is the unique Borel probability measure
on $\mathbb{R}^{d}$ with,
\[
\mu_{t}=\sum_{j=0}^{m}p_{j}\cdot\varphi_{t,j}\mu_{t}\:.
\]
By \cite[Theorem 2.8]{FH} it follows that $\mu_{t}$ is exact dimensional.
That is there exists a number $\dim\mu_{t}\in[0,d]$ such that,
\[
\dim\mu_{t}=\underset{\delta\downarrow0}{\lim}\:\frac{\log\mu(B(x,\delta))}{\log\delta}\text{ for }\mu\text{-a.e. }x\in\mathbb{R}^{d},
\]
where $B(x,\delta)$ is the open ball in $\mathbb{R}^{d}$ with centre
$x$ and radius $\delta$.

Denote the set $\{0,...,m\}$ by $\Lambda$. Given $n\ge1$, $j_{1}...j_{n}=w\in\Lambda^{n}$
and $t\in(\mathbb{R}^{d})^{m}$ we shall write $\varphi_{t,w}$, $p_{w}$
and $\lambda_{w}$ in place of,
\[
\varphi_{t,j_{1}}\circ...\circ\varphi_{t,j_{n}},\:p_{j_{1}}\cdot...\cdot p_{j_{n}}\text{ and }\lambda_{j_{1}}\cdot...\cdot\lambda_{j_{n}}\:.
\]
The IFS $\Phi_{t}$ is said to have exact overlaps if there exist
$n\ge1$ and distinct words $w_{1},w_{2}\in\Lambda^{n}$ with $\varphi_{t,w_{1}}=\varphi_{t,w_{2}}$.

For $d\ge1$ denote by $\mathbf{G}^{d}$ the group of all affine transformations
$\psi:\mathbb{R}^{d}\rightarrow\mathbb{R}^{d}$ for which there exist
$0\ne\lambda'\in\mathbb{R}$ and $t'\in\mathbb{R}^{d}$ with $\psi(x)=\lambda'x+t'$
for all $x\in\mathbb{R}^{d}$. Given such a $\psi\in\mathbf{G}^{d}$
we sometimes write $\lambda_{\psi}$ for $\lambda'$ and $t_{\psi}$
for $t'$. For $t\in(\mathbb{R}^{d})^{m}$ and $n\ge1$ set,
\[
\nu_{t}^{(n)}=\sum_{w\in\Lambda^{n}}p_{w}\delta_{\varphi_{t,w}},
\]
where $\delta_{\varphi_{t,w}}$ is the Dirac measure at $\varphi_{t,w}$.
Thus $\nu_{t}^{(n)}$ is a finitely supported probability measure
on $\mathbf{G}^{d}$.

Suppose that $Y_{t,1},Y_{t,2},...$ are i.i.d. $\mathbf{G}^{d}$-valued
random variables with,
\[
\mathbb{P}\{Y_{t,1}=\varphi_{t,j}\}=p_{j}\text{ for each }j\in\Lambda\:.
\]
Then for each $n\ge1$ the distribution of $Y_{t,1}\cdot...\cdot Y_{t,n}$
is equal to $\nu_{t}^{(n)}$. From this it follows that,
\begin{equation}
\nu_{t}^{(n+k)}=\nu_{t}^{(n)}*\nu_{t}^{(k)}\text{ for all }n,k\ge1,\label{eq:=00003Dto convolution}
\end{equation}
where $\nu_{t}^{(n)}*\nu_{t}^{(k)}$ is the convolution of $\nu_{t}^{(n)}$
with $\nu_{t}^{(k)}$.

\subsection{Partitions, entropy and dimension}

Given $d\ge1$ and $n\ge0$ write $\mathcal{D}_{n}^{d}$ for the level-$n$
dyadic partition of $\mathbb{R}^{d}$, that is
\[
\mathcal{D}_{n}^{d}=\{[\frac{k_{1}}{2^{n}},\frac{k_{1}+1}{2^{n}})\times...\times[\frac{k_{d}}{2^{n}},\frac{k_{d}+1}{2^{n}})\::\:k_{1},...,k_{d}\in\mathbb{Z}\}\:.
\]
Denote by $\mathcal{E}_{n}^{d}$ the level-$n$ dyadic partition of
$\mathbf{G}^{d}$ according to the translation part, i.e.
\[
\mathcal{E}_{n}^{d}=\left\{ \{\psi\in\mathbf{G}^{d}\::\:t_{\psi}\in D\}\::\:D\in\mathcal{D}_{n}^{d}\right\} \:.
\]
For a real $r\ge0$ we write $\mathcal{D}_{r}^{d}$ and $\mathcal{E}_{r}^{d}$
instead of $\mathcal{D}_{\left\lfloor r\right\rfloor }^{d}$ and $\mathcal{E}_{\left\lfloor r\right\rfloor }^{d}$.
Let $\mathcal{F}^{d}$ be the partition of $\mathbf{G}^{d}$ according
to the scaling part, that is
\[
\mathcal{F}^{d}=\left\{ \{\psi\in\mathbf{G}^{d}\::\:\lambda_{\psi}=\lambda'\}\::\:\lambda'\in\mathbb{R}\setminus\{0\}\right\} \:.
\]
When using these notations we shall omit the superscript $d$ whenever
it is clear from the context.

Given a measurable space $X$, a measurable partition of it $\mathcal{D}$
and a probability measure $\theta$ on $X$, we write $H(\theta,\mathcal{D})$
for the entropy of $\theta$ with respect to $\mathcal{D}$. That
is,
\[
H(\theta,\mathcal{D})=-\sum_{D\in\mathcal{D}}\theta(D)\log\theta(D)\:.
\]
When $\mathcal{D}$ is the partition of $X$ into singletons we write
$H(\theta)$ instead of $H(\theta,\mathcal{D})$.

The following lemma is well known. Variations of it has appeared before
in \cite{Ho1} and \cite{Ho2}. We include a short proof of it here
for completeness. During the proof we use freely some basic properties
of entropy which can be found for instance in \cite[Section 3.1]{Ho1}. 
\begin{lem}
\label{lem:approx of dim}Let $d\ge1$ and $t\in(\mathbb{R}^{d})^{m}$
be given, then
\[
\dim\mu_{t}=\underset{n}{\lim}\:\frac{1}{\chi n}H(\nu_{t}^{(n)},\mathcal{E}_{\chi n})\:.
\]
\end{lem}

\begin{proof}
Let $\epsilon>0$ and let $n\ge1$ be large with respect to $\epsilon$.
Let $\Pi,\Pi_{n}:\Lambda^{\mathbb{N}}\rightarrow\mathbb{R}^{d}$ be
such that for $(\omega_{k})_{k\ge0}=\omega\in\Lambda^{\mathbb{N}}$,
\[
\Pi\omega=\underset{k}{\lim}\:\varphi_{t,\omega_{0}...\omega_{k}}(0)\text{ and }\Pi_{n}\omega=\varphi_{t,\omega_{0}...\omega_{n-1}}(0)\:.
\]
Denote by $\xi$ the Bernoulli measure on $\Lambda^{\mathbb{N}}$
corresponding to $p$, i.e. $\xi=p^{\mathbb{N}}$. Note that $\Pi\xi=\mu_{t}$
and,
\begin{equation}
H(\nu_{t}^{(n)},\mathcal{E}_{\chi n})=H(\Pi_{n}\xi,\mathcal{D}_{\chi n})\:.\label{eq:H(nu_t^n) =00003D}
\end{equation}
By the law of large numbers, and by assuming $n$ is large enough
with respect to $\epsilon$, there exists a Borel set $E\subset\Lambda^{\mathbb{N}}$
with $\xi(E)>1-\epsilon$ and,
\begin{equation}
|\frac{1}{n}\log|\lambda_{\omega_{0}...\omega_{n-1}}|+\chi|<\epsilon\text{ for each }\omega\in E\:.\label{eq:close to Lyapunov}
\end{equation}

For a Borel set $F\subset\Lambda^{\mathbb{N}}$ with $\xi(F)>0$ write
$\xi_{F}=\frac{1}{\xi(F)}\xi|_{F}$, where $\xi|_{F}$ is the restriction
of $\xi$ to $F$. By the concavity of entropy and since $H(\Pi\xi_{E},\mathcal{D}_{\chi n})=O_{\lambda,t}(n)$,
\begin{eqnarray*}
H(\mu_{t},\mathcal{D}_{\chi n}) & \ge & \xi(E)H(\Pi\xi_{E},\mathcal{D}_{\chi n})+\xi(E^{c})H(\Pi\xi_{E^{c}},\mathcal{D}_{\chi n})\\
 & \ge & H(\Pi\xi_{E},\mathcal{D}_{\chi n})-O_{\lambda,t}(\epsilon n)\:.
\end{eqnarray*}
Similarly, by the convexity bound for entropy and since $H(\Pi\xi_{E^{c}},\mathcal{D}_{\chi n})=O_{\lambda,t}(n)$
(assuming $\xi(E^{c})>0$),
\begin{eqnarray*}
H(\mu_{t},\mathcal{D}_{\chi n}) & \le & \xi(E)H(\Pi\xi_{E},\mathcal{D}_{\chi n})+\xi(E^{c})H(\Pi\xi_{E^{c}},\mathcal{D}_{\chi n})+1\\
 & \le & H(\Pi\xi_{E},\mathcal{D}_{\chi n})+O_{\lambda,t}(\epsilon n)\:.
\end{eqnarray*}
Thus,
\begin{equation}
H(\mu_{t},\mathcal{D}_{\chi n})=H(\Pi\xi_{E},\mathcal{D}_{\chi n})+O_{\lambda,t}(\epsilon n),\label{eq:H(mu) close to rst}
\end{equation}
and a similar argument gives,
\begin{equation}
H(\Pi_{n}\xi,\mathcal{D}_{\chi n})=H(\Pi_{n}\xi_{E},\mathcal{D}_{\chi n})+O_{\lambda,t}(\epsilon n)\:.\label{eq:H(Pixi) close to rst}
\end{equation}

Note that by (\ref{eq:close to Lyapunov}),
\[
|\Pi\omega-\Pi_{n}\omega|=O_{\lambda,t}(2^{-n(\chi-\epsilon)})\text{ for all }\omega\in E,
\]
from which it follows easily that,
\[
H(\Pi\xi_{E},\mathcal{D}_{\chi n})=H(\Pi_{n}\xi_{E},\mathcal{D}_{\chi n})+O_{\lambda,t}(\epsilon n)\:.
\]
From this, (\ref{eq:H(nu_t^n) =00003D}), (\ref{eq:H(mu) close to rst})
and (\ref{eq:H(Pixi) close to rst}),
\[
H(\mu_{t},\mathcal{D}_{\chi n})=H(\nu_{t}^{(n)},\mathcal{E}_{\chi n})+O_{\lambda,t}(\epsilon n)\:.
\]
Additionally, since $\mu_{t}$ is exact dimensional we may assume,
\[
|\frac{1}{\chi n}H(\mu_{t},\mathcal{D}_{\chi n})-\dim\mu_{t}|<\epsilon\:.
\]
Hence,
\[
\frac{1}{\chi n}H(\nu_{t}^{(n)},\mathcal{E}_{\chi n})=\dim\mu_{t}+O_{\lambda,t}(\epsilon),
\]
which completes the proof of the lemma.
\end{proof}
\begin{cor}
\label{cor:ub on dim}Let $d\ge1$ and $t\in(\mathbb{R}^{d})^{m}$
be given, then
\[
\dim\mu_{t}\le\frac{1}{\chi n}H(\nu_{t}^{(n)})\text{ for all }n\ge1\:.
\]
\end{cor}

\begin{proof}
By (\ref{eq:=00003Dto convolution}) it follows that for every $n,k\ge1$,
\[
H(\nu_{t}^{(n+k)})=H(\nu_{t}^{(n)}*\nu_{t}^{(k)})\le H(\nu_{t}^{(n)})+H(\nu_{t}^{(k)})\:.
\]
Thus by the Fekete lemma for sub-additive sequences,
\begin{equation}
\underset{n}{\lim}\:\frac{1}{n}H(\nu_{t}^{(n)})=\underset{n}{\inf}\:\frac{1}{n}H(\nu_{t}^{(n)})\:.\label{eq:lim=00003Dinf}
\end{equation}
Now by Lemma \ref{lem:approx of dim},
\[
\dim\mu_{t}=\underset{n}{\lim}\:\frac{1}{\chi n}H(\nu_{t}^{(n)},\mathcal{E}_{\chi n})\le\underset{n}{\lim}\:\frac{1}{\chi n}H(\nu_{t}^{(n)}),
\]
and the corollary follows from (\ref{eq:lim=00003Dinf}).
\end{proof}

\subsection{Affine irreducibility}

Given $d\ge1$ and $t\in(\mathbb{R}^{d})^{m}$ we say that $\Phi_{t}$
is affinely irreducible if there is no proper affine subspace $V$
of $\mathbb{R}^{d}$ with $\varphi_{t,j}(V)=V$ for all $0\le j\le m$.
\begin{lem}
\label{lem:affine irreducibility}Let $1\le d\le m$ and $(t_{j})_{j=1}^{m}=t\in\mathbb{R}^{dm}$
be with,
\[
\mathrm{span}\{t_{1},...,t_{m}\}=\mathbb{R}^{d}\:.
\]
Then $\Phi_{t}$ is affinely irreducible.
\end{lem}

\begin{proof}
Without loss of generality we may assume that $t_{1},...,t_{d}$ are
linearly independent. For $1\le j\le d$ let $x_{j}$ be the fixed
point of $\varphi_{t,j}$, then $x_{j}\ne0$ and $x_{j}\in K_{t}\cap\mathrm{span}\{t_{j}\}$.
Since $\varphi_{t,0}(0)=0$ we also have $0\in K_{t}$. From these
facts it follows that the affine span of $K_{t}$ is equal to $\mathbb{R}^{d}$.

On the other hand, if $s\in(\mathbb{R}^{d})^{m}$ is such that $\Phi_{s}$
is not affinely irreducible then there exists a proper affine subspace
$V$ of $\mathbb{R}^{d}$ with $\varphi_{s,j}(V)=V$ for all $0\le j\le m$.
From this it clearly follows that $K_{s}\subset V$, which implies
that the affine span of $K_{s}$ is contained in $V$. Since $V$
is proper we must have $t\ne s$, which completes the proof of the
lemma.
\end{proof}

\subsection{A Theorem of Hochman}

The following statement follows almost directly from a theorem of
Hochman regarding self-similar measures in $\mathbb{R}^{d}$. Given
two partitions $\mathcal{C}_{1}$ and $\mathcal{C}_{2}$ of some space
$X$ we denote by $\mathcal{C}_{1}\vee\mathcal{C}_{2}$ their common
refinement, that is
\[
\mathcal{C}_{1}\vee\mathcal{C}_{2}=\{C_{1}\cap C_{2}\::\:C_{1}\in\mathcal{C}_{1}\text{ and }C_{2}\in\mathcal{C}_{2}\}\:.
\]
 
\begin{thm}
\label{thm:follows from =00005BHo2=00005D}Let $d\ge1$ and $t\in(\mathbb{R}^{d})^{m}$
be given. Suppose that $\Phi_{t}$ is affinely irreducible and that
$\dim\mu_{t}<\beta$. Then for every $q\ge\chi$,
\[
\underset{n}{\lim}\:\frac{1}{\chi n}H(\nu_{t}^{(n)},\mathcal{E}_{qn}\vee\mathcal{F})=\dim\mu_{t}\:.
\]
\end{thm}

\begin{proof}
Given a linear subspace $V$ of $\mathbb{R}^{d}$ let $\pi_{V}$ be
the orthogonal projection onto $V$. Denote by $\{\mu_{t,x}^{V}\}_{x\in\mathbb{R}^{d}}$
the disintegration of $\mu_{t}$ with respect to $\pi_{V^{\perp}}^{-1}(\mathcal{B})$,
where $V^{\perp}$ is the orthogonal complement of $V$ and $\mathcal{B}$
is the Borel $\sigma$-algebra of $\mathbb{R}^{d}$. By \cite[Theorem 3.1]{FJ}
it follows that $\mu_{t,x}^{V}$ is exact dimensional for $\mu_{t}$-a.e.
$x\in\mathbb{R}^{d}$.

Since $\Phi_{t}$ is affinely irreducible it follows by \cite[Theorem 1.5]{Ho2}
that at least one of the following three alternatives is satisfied,
\begin{enumerate}
\item $\dim\mu=\min\{d,H(p)/\chi\}$;
\item $\underset{n}{\lim}\:\frac{1}{\chi n}\left(H(\nu_{t}^{(n)},\mathcal{E}_{qn})-H(\nu_{t}^{(n)},\mathcal{E}_{\chi n})\right)=0$
for all $q\ge\chi$;
\item there exists a linear subspace $V$ of $\mathbb{R}^{d}$ with $\dim V>0$
and,
\[
\dim\mu_{t,x}^{V}=\dim V\text{ for }\mu\text{-a.e. }x\in\mathbb{R}^{d}\:.
\]
\end{enumerate}
Since $\dim\mu_{t}<\beta$ the first option is clearly not possible.
By \cite[Corollary 3.1]{FJ} it follows that,
\[
\dim\mu_{t}\ge\dim\mu_{t,x}^{V}\text{ for }\mu\text{-a.e. }x\in\mathbb{R}^{d}\:.
\]
From this and $\beta\le1$ it follows that the third alternative is
also not possible. Thus the second alternative must hold, and so by
Lemma \ref{lem:approx of dim} it follows that for every $q\ge\chi$,
\begin{equation}
\underset{n}{\lim}\:\frac{1}{\chi n}H(\nu_{t}^{(n)},\mathcal{E}_{qn})=\underset{n}{\lim}\:\frac{1}{\chi n}H(\nu_{t}^{(n)},\mathcal{E}_{\chi n})=\dim\mu_{t}\:.\label{eq:ent(E_qn)=00003Ddim}
\end{equation}

Additionally, for each $n\ge1$ the cardinality of the set $\{\lambda_{w}\::\:w\in\Lambda^{n}\}$
is at most $n^{m+1}$. Hence,
\[
H(\nu_{t}^{(n)},\mathcal{E}_{qn}\vee\mathcal{F})\le H(\nu_{t}^{(n)},\mathcal{E}_{qn})+H(\nu_{t}^{(n)},\mathcal{F})\le H(\nu_{t}^{(n)},\mathcal{E}_{qn})+\log n^{m+1},
\]
which implies,
\[
\underset{n}{\lim}\:\frac{1}{\chi n}\left(H(\nu_{t}^{(n)},\mathcal{E}_{qn}\vee\mathcal{F})-H(\nu_{t}^{(n)},\mathcal{E}_{qn})\right)=0\:.
\]
This together with (\ref{eq:ent(E_qn)=00003Ddim}) completes the proof
of the theorem.
\end{proof}

\subsection{Lower semicontinuity of dimension}

The following theorem follows directly from \cite[Theorem 1.8]{Fe},
where it is proven in much greater generality.
\begin{thm}
\label{thm:lower semi-cont of dim}For every $d\ge1$ the function
which takes $t\in(\mathbb{R}^{t})^{m}$ to $\dim\mu_{t}$ is lower
semi-continuous.
\end{thm}

\subsection{Notations for polynomials and an exponential lower bound}

Let $\mathbb{N}$ be the set of non-negative integers. For a multi-index
$(\alpha_{0},...,\alpha_{m})=\alpha\in\mathbb{N}^{m+1}$ write,
\[
X^{\alpha}=X_{0}^{\alpha_{0}}\cdot...\cdot X_{m}^{\alpha_{m}},\:\lambda^{\alpha}=\lambda_{0}^{\alpha_{0}}\cdot...\cdot\lambda_{m}^{\alpha_{m}}\text{ and }|\alpha|=\alpha_{0}+...+\alpha_{m},
\]
where $X_{0},...,X_{m}$ are formal variables and $\lambda=(\lambda_{j})_{j=0}^{m}$
was fixed in Section \ref{subsec:Some-notations}. Given a polynomial
\[
\sum_{\alpha\in\mathbb{N}^{m+1}}c_{\alpha}X^{\alpha}=P(X)\in\mathbb{Z}[X_{0},...,X_{m}]
\]
denote by $\deg(P)$ the total degree of $P$, that is
\[
\deg(P)=\max\{|\alpha|\::\:\alpha\in\mathbb{N}^{m+1}\text{ and }c_{\alpha}\ne0\}\:.
\]
Also set,
\[
\ell_{\infty}(P)=\underset{\alpha\in\mathbb{N}^{m+1}}{\max}\:|c_{k}|\text{ and }\ell_{1}(P)=\sum_{\alpha\in\mathbb{N}^{m+1}}|c_{k}|\:.
\]
For $l,n\ge1$ write,
\[
\mathcal{P}(l,n)=\{P\in\mathbb{Z}[X_{0},...,X_{m}]\::\:\ell_{\infty}(P)\le l\text{ and }\deg(P)<n\}\:.
\]
Note that for $P_{1},...,P_{k}\in\mathbb{Z}[X_{0},...,X_{m}]$,
\[
\ell_{\infty}(\Pi_{j=1}^{k}P_{j})\le\ell_{1}(\Pi_{j=1}^{k}P_{j})\le\Pi_{j=1}^{k}\ell_{1}(P_{j})\:.
\]
From this and since $\ell_{1}(P)\le ln^{m+1}$ for $P\in\mathcal{P}(l,n)$,
\begin{equation}
\Pi_{j=1}^{k}P_{j}\in\mathcal{P}(l^{k}n^{(m+1)k},kn)\text{ for every }P_{1},...,P_{k}\in\mathcal{P}(l,n)\:.\label{eq:control on product of poly}
\end{equation}

The following lemma will be used many times during the proof of Theorem
\ref{thm:main thm inductive}. It is the only place in which the fact
that $\lambda_{0},...,\lambda_{m}$ are algebraic is used.
\begin{lem}
\label{lem:lb on val of poly}There exists a constant $M=M(m,\lambda)>1$
such that for every $l,n\ge1$ and $P\in\mathcal{P}(l,n)$ we have
$P(\lambda)=0$ or $|P(\lambda)|\ge l^{-M}M^{-n}$.
\end{lem}

\begin{proof}
Given an algebraic $\eta\in\mathbb{C}$ write $H(\eta)$ for its hight,
as defined in \cite[Chapter 14]{Mas}. Write,
\[
H=\underset{0\le j\le m}{\max}\:H(\lambda_{j}),
\]
and denote by $D$ the degree of the field extension $\mathbb{Q}[\lambda_{0},...,\lambda_{m}]/\mathbb{Q}$.

Let $l,n\ge1$ and $P\in\mathcal{P}(l,n)$ be given. Since $P\in\mathcal{P}(l,n)$
we have $\ell_{1}(P)\le ln^{m+1}$. Thus by \cite[Proposition 14.7]{Mas},
\[
H(P(\lambda))\le\ell_{1}(P)\prod_{j=0}^{m}H(\lambda_{j})^{n}\le ln^{m+1}\cdot H^{(m+1)n}\:.
\]
Since $P(\lambda)\in\mathbb{Q}[\lambda_{0},...,\lambda_{m}]$ it follows
by \cite[Proposition 14.13]{Mas} that $P(\lambda)=0$ or,
\[
|P(\lambda)|\ge H(P(\lambda))^{-D}\ge(ln^{m+1}\cdot H^{(m+1)n})^{-D}\:.
\]
The lemma now follows by taking $M$ to be sufficiently large with
respect to $m$, $H$ and $D$.
\end{proof}

\subsection{A family of linear functionals}

Recall that according to the notation introduced in Section \ref{subsec:Some-notations}
we have that $\Phi_{t}=\{\varphi_{t,j}\}_{j=0}^{m}$ is an IFS on
$\mathbb{R}$ for each $t\in\mathbb{R}^{m}$. Given $n\ge1$ and $w_{1},w_{2}\in\Lambda^{n}$
let $L_{w_{1},w_{2}}:\mathbb{R}^{m}\rightarrow\mathbb{R}$ be such
that,
\[
L_{w_{1},w_{2}}(t)=\varphi_{t,w_{1}}(0)-\varphi_{t,w_{2}}(0)\text{ for each }t\in\mathbb{R}^{m}\:.
\]
Note that $L_{w_{1},w_{2}}$ is a linear functional on $\mathbb{R}^{m}$.
Set,
\[
\mathcal{L}_{n}=\{L_{w_{1},w_{2}}\::\:w_{1},w_{2}\in\Lambda^{n},\:w_{1}\ne w_{2}\text{ and }\lambda_{w_{1}}=\lambda_{w_{2}}\}\:.
\]

\begin{lem}
\label{lem:char of L}Let $n\ge1$ and $L\in\mathcal{L}_{n}$ be given.
Then there exist $P^{1},...,P^{m}\in\mathcal{P}(1,n)$ such that,
\[
L(t)=\sum_{j=1}^{m}P^{j}(\lambda)t_{j}\text{ for each }(t_{1},...,t_{m})=t\in\mathbb{R}^{m}\:.
\]
\end{lem}

\begin{proof}
Given a word $q_{0}...q_{n-1}=w\in\Lambda^{n}$ write $\alpha(w)$
for the multi-index $(\alpha_{0},...,\alpha_{m})\in\mathbb{N}^{m+1}$
with,
\[
\alpha_{j}=\#\{0\le k<n\::\:q_{k}=j\}\text{ for each }0\le j\le m\:.
\]
Let $q_{0}...q_{n-1}=w_{1}\in\Lambda^{n}$ and $l_{0}...l_{n-1}=w_{2}\in\Lambda^{n}$
be with $L=L_{w_{1},w_{2}}$. Let $(t_{1},...,t_{m})=t\in\mathbb{R}^{m}$
and write $t_{0}=0$, then
\[
\varphi_{t,w_{1}}(0)=\sum_{k=0}^{n-1}\lambda_{q_{0}...q_{k-1}}t_{q_{k}}=\sum_{j=1}^{m}t_{j}\sum_{k=0}^{n-1}1_{\{q_{k}=j\}}\lambda^{\alpha(q_{0}...q_{k-1})},
\]
and similarly for $\varphi_{t,w_{2}}(0)$ with $l_{k}$ in place of
$q_{k}$. Thus,
\begin{eqnarray*}
L(t) & = & \varphi_{t,w_{1}}(0)-\varphi_{t,w_{2}}(0)\\
 & = & \sum_{j=1}^{m}t_{j}\sum_{k=0}^{n-1}\left(\lambda^{\alpha(q_{0}...q_{k-1})}1_{\{q_{k}=j\}}-\lambda^{\alpha(l_{0}...l_{k-1})}1_{\{l_{k}=j\}}\right)\:.
\end{eqnarray*}
Now since for each $1\le j\le m$,
\[
\sum_{k=0}^{n-1}\left(X^{\alpha(q_{0}...q_{k-1})}1_{\{q_{k}=j\}}-X^{\alpha(l_{0}...l_{k-1})}1_{\{l_{k}=j\}}\right)\in\mathcal{P}(1,n),
\]
the proof is complete.
\end{proof}

\section{\label{sec:Proof-of-Theorem}Proof of Theorem \ref{thm:main thm inductive}}

Let us recall the statement of Theorem \ref{thm:main thm inductive}.
\begin{thm*}
For every $1\le d<m$ the following holds. Let $(t_{j})_{j=1}^{m}=t\in\mathbb{R}^{dm}$
be such that,
\[
\mathrm{span}\{t_{1},...,t_{m}\}=\mathbb{R}^{d}\:.
\]
Supposed that $\Phi_{t}$ has no exact overlaps, then $\dim\mu_{t}\ge\beta$.
\end{thm*}
\begin{proof}
The proof is carried out by backward induction on $d$. Let $1\le d<m$
and suppose that the theorem has been proven for each $d<d'<m$. Let
$(t_{j})_{j=1}^{m}=t\in\mathbb{R}^{dm}$ be such that,
\[
\mathrm{span}\{t_{1},...,t_{m}\}=\mathbb{R}^{d}\text{ and }\dim\mu_{t}<\beta\:.
\]
We need to show that $\Phi_{t}$ has exact overlaps.

Let $\{e_{1},...,e_{d}\}$ be the standard basis of $\mathbb{R}^{d}$.
For every invertible linear transformation $T:\mathbb{R}^{d}\rightarrow\mathbb{R}^{d}$,
\[
\Phi_{T(t)}=\{T\circ\varphi_{t,j}\circ T^{-1}\}_{j=0}^{m}\:.
\]
Thus, without loss of generality we may assume that,
\begin{equation}
\{e_{1},...,e_{d}\}\subset\{t_{1},...,t_{m}\}\:.\label{eq:st basis contained}
\end{equation}
There are two cases to consider during the proof. The containment
(\ref{eq:st basis contained}) will only be used when dealing with
the first case.

Let $0<\epsilon<\frac{1}{2}(\beta-\dim\mu_{t})$ and let $\delta>0$
be small with respect to $m$, $p$, $\lambda$ and $t$. In particular,
$\delta$ is assumed to be small with respect to the constant $M=M(m,\lambda)>1$
obtained in Lemma \ref{lem:lb on val of poly}. By Lemma \ref{lem:affine irreducibility}
it follows that $\Phi_{t}$ is affinely irreducible. Since $\delta$
is small with respect to $p$ and $\lambda$ we may assume that $-\log\delta>\chi$.
Thus, by Theorem \ref{thm:follows from =00005BHo2=00005D} there exists
$N\ge1$ such that,
\begin{equation}
\frac{1}{\chi n}H(\nu_{t}^{(n)},\mathcal{E}_{\left\lceil -\log\delta\right\rceil n}\vee\mathcal{F})<\dim\mu_{t}+\epsilon\text{ for all }n\ge N\:.\label{eq:small ent}
\end{equation}

For every $1\le j\le m$ let $t_{j}^{1},...,t_{j}^{d}\in\mathbb{R}$
be such that $t_{j}=(t_{j}^{l})_{l=1}^{d}$. For $1\le l\le d$ set
$t^{l}=(t_{j}^{l})_{j=1}^{m}$. Given $n\ge N$ let $\mathcal{A}_{n}$
be the collection of all $L\in\mathcal{L}_{n}$ with $|L(t^{l})|\le\delta^{n}$
for each $1\le l\le d$. By (\ref{eq:small ent}) and since,
\[
\chi n(\dim\mu_{t}+\epsilon)<\chi n(\beta-\epsilon)<nH(p),
\]
there exist $w_{1},w_{2}\in\Lambda^{n}$ such that $w_{1}\ne w_{2}$
but $\varphi_{t,w_{1}}$ and $\varphi_{t,w_{2}}$ belong to the same
atom of the partition $\mathcal{E}_{\left\lceil -\log\delta\right\rceil n}\vee\mathcal{F}$.
Thus $\lambda_{w_{1}}=\lambda_{w_{2}}$ and for each $1\le l\le d$,
\[
|L_{w_{1},w_{2}}(t^{l})|=|\varphi_{t^{l},w_{1}}(0)-\varphi_{t^{l},w_{2}}(0)|\le\delta^{n}\:.
\]
It follows that $L_{w_{1},w_{2}}\in\mathcal{A}_{n}$ and in particular
that $\mathcal{A}_{n}$ is nonempty.

Write $q_{n}$ for $|\mathcal{A}_{n}|$ and let $\{L_{n,i}\}_{i=1}^{q_{n}}$
be an enumeration of $\mathcal{A}_{n}$. By Lemma \ref{lem:char of L}
for every $1\le i\le q_{n}$ there exist $P_{n,i}^{1},...,P_{n,i}^{m}\in\mathcal{P}(1,n)$
such that,
\[
L_{n,i}(x)=\sum_{j=1}^{m}P_{n,i}^{j}(\lambda)x_{j}\text{ for every }(x_{1},...,x_{m})=x\in\mathbb{R}^{m}\:.
\]
For $1\le j\le m$ set $a_{n,i}^{j}=P_{n,i}^{j}(\lambda)$, then $a_{n,i}^{j}=O_{\lambda,m}(1)$
since $\lambda_{0},...,\lambda_{m}\in(-1,1)$ and $P_{n,i}^{j}\in\mathcal{P}(1,n)$.
Write $a_{n,i}=(a_{n,i}^{1},...,a_{n,i}^{m})$ and $B_{n}=\{1,...,q_{n}\}\times\{1,...,m\}$.
Denote by $r_{n}$ the rank of the matrix $(a_{n,i}^{j})_{(i,j)\in B_{n}}$.
For each $1\le i\le q_{n}$ and $1\le l\le d$ set $\rho_{n,i}^{l}=L_{n,i}(t^{l})$,
then $|\rho_{n,i}^{l}|\le\delta^{n}$ by the definition of $\mathcal{A}_{n}$.

If $a_{n,i}=0$ for some $n\ge N$ and $1\le i\le q_{n}$ then $L_{n,i}(t^{l})=0$
for each $1\le l\le d$. By the definition of $\mathcal{L}_{n}$ this
clearly implies that $\Phi_{t}$ has exact overlaps. Thus we may assume
that $a_{n,i}\ne0$, and in particular that $1\le r_{n}\le m$. In
fact we can get a better upper bound.
\begin{claim*}
For $n\ge N$ we have $r_{n}\le m-d$.
\end{claim*}
\begin{proof}
Without loss of generality we may assume that $\{a_{n,i}\}_{i=1}^{r_{n}}$
are linearly independent. Write $V=\mathrm{span}\{a_{n,1},...,a_{n,r_{n}}\}$
and let $V^{\perp}$ be the orthogonal complement of $V$ in $\mathbb{R}^{m}$.
Denote the orthogonal projections onto $V^{\perp}$ by $\pi_{V^{\perp}}$.

Given $k\ge1$ and vectors $x_{1},...,x_{k}\in\mathbb{R}^{m}$ write
$G(x_{1},...,x_{k})$ for their Gram determinant. That is,
\[
G(x_{1},...,x_{k})=\det\left(\bigl\langle x_{i},x_{j}\bigr\rangle\::\:1\le i,j\le k\right),
\]
where $\left\langle \cdot,\cdot\right\rangle $ is the standard inner
product of $\mathbb{R}^{m}$. Note that $G(x_{1},...,x_{k})$ is equal
to the square of the $k$-dimensional volume of the parallelotope
spanned by $x_{1},...,x_{k}$. Since $\{a_{n,i}\}_{i=1}^{r_{n}}$
are linearly independent we have $G(a_{n,1},...,a_{n,r_{n}})\ne0$.
Thus for every $x\in\mathbb{R}^{m}$,
\begin{equation}
\Vert\pi_{V^{\perp}}x\Vert^{2}=\frac{G(a_{n,1},...,a_{n,r_{n}},x)}{G(a_{n,1},...,a_{n,r_{n}})}\:.\label{eq:formula for dist}
\end{equation}

Recall that $a_{n,i}^{j}=P_{n,i}^{j}(\lambda)$ with $P_{n,i}^{j}\in\mathcal{P}(1,n)$
for each $(i,j)\in B_{n}$. Thus, by (\ref{eq:control on product of poly})
it follows that for every $1\le i_{1},i_{2}\le r_{n}$ there exists
$P\in\mathcal{P}(mn^{2(m+1)},2n)$ with $\bigl\langle a_{n,i_{1}},a_{n,i_{2}}\bigr\rangle=P(\lambda)$.
Since $r_{n}\le m$ another application of (\ref{eq:control on product of poly})
shows that there exists,
\[
D\in\mathcal{P}(m^{2m}n^{4m(m+1)},2mn),
\]
such that,
\[
G(a_{n,1},...,a_{n,r_{n}})=D(\lambda)\:.
\]
From this, $G(a_{n,1},...,a_{n,r_{n}})\ne0$ and Lemma \ref{lem:lb on val of poly}
we obtain,
\[
G(a_{n,1},...,a_{n,r_{n}})\ge(m^{2m}n^{4m(m+1)})^{-M}\cdot M^{-2mn}\:.
\]
Thus by assuming that $\delta$ is small enough with respect to $\lambda$
and $m$,
\begin{equation}
G(a_{n,1},...,a_{n,r_{n}})\ge\delta^{n/2}\:.\label{eq:lb on gram det}
\end{equation}

Since $a_{n,i}^{j}=O_{\lambda,m}(1)$ for $(i,j)\in B_{n}$,
\begin{equation}
\bigl\langle a_{n,i_{1}},a_{n,i_{2}}\bigr\rangle=O_{\lambda,m}(1)\text{ for each }1\le i_{1},i_{2}\le r_{n}\:.\label{eq:ub on inner products}
\end{equation}
For every $1\le l\le d$ and $1\le i\le r_{n}$,
\[
\delta^{n}\ge|L_{n,i}(t^{l})|=\left|\sum_{j=1}^{m}a_{n,i}^{j}t_{j}^{l}\right|=\left|\bigl\langle a_{n,i},t^{l}\bigr\rangle\right|\:.
\]
From this and (\ref{eq:ub on inner products}),
\[
G(a_{n,1},...,a_{n,r_{n}},t^{l})=\Vert t^{l}\Vert^{2}\cdot G(a_{n,1},...,a_{n,r_{n}})+O_{\lambda,m}(\delta^{n})\:.
\]
Hence from (\ref{eq:lb on gram det}) and (\ref{eq:formula for dist}),
\[
\Vert\pi_{V^{\perp}}t^{l}\Vert^{2}=\Vert t^{l}\Vert^{2}+O_{\lambda,m}(\delta^{n/2})\:.
\]

Thus there exist vectors $u_{1},...,u_{d}\in V^{\perp}$ with $\Vert t^{l}-u_{l}\Vert=O_{\lambda,m}(\delta^{n/4})$
for each $1\le l\le d$. From $\mathrm{span}\{t_{1},...,t_{m}\}=\mathbb{R}^{d}$
it follows that $\{t^{l}\}_{l=1}^{d}$ are linearly independent. Hence
by assuming that $\delta$ is small enough with respect to $\lambda$,
$m$ and $t$, we get that $\{u_{l}\}_{l=1}^{d}$ are also linearly
independent. From this it follows that,
\[
d\le\dim V^{\perp}=m-\dim V=m-r_{n},
\]
which completes the proof of the claim.
\end{proof}
Now there are two cases to consider.

\subsection*{First case}

Suppose first that $\underset{n}{\liminf}\:r_{n}=m-d$. By the last
claim and by increasing $N$ without changing the notation we may
assume that $r_{n}=m-d$ for all $n\ge N$. Recall that $\{e_{j}\}_{j=1}^{d}\subset\{t_{j}\}_{j=1}^{m}$,
and so we may assume for simplicity that,
\begin{equation}
t_{j}=e_{j}\text{ for each }1\le j\le d\:.\label{eq:basis at beginning}
\end{equation}
The assumption (\ref{eq:basis at beginning}) will only be made when
dealing with the present case, in which it will be clear that it makes
no difference.

Set $J=\{d+1,...,m\}$ and $I=\{1,...,m-d\}$. Without loss of generality
we may assume that $\{a_{n,i}\::\:i\in I\}$ are linearly independent
for each $n\ge N$. Given $J_{0}\subset\{1,...,m\}$ with $|J_{0}|=m-d$
write $P_{n}^{J_{0}}(X)$ for the determinant of the matrix $(P_{n,i}^{j}(X))_{(i,j)\in I\times J_{0}}$.
Note that $P_{n}^{J_{0}}(X)\in\mathbb{Z}[X_{0},...,X_{m}]$ and that
$P_{n}^{J_{0}}(\lambda)$ is equal to the determinant of $(a_{n,i}^{j})_{(i,j)\in I\times J_{0}}$.
\begin{claim*}
For every $n\ge N$ we have $P_{n}^{J}(\lambda)\ne0$.
\end{claim*}
\begin{proof}
Assume by contradiction that there exists $n\ge N$ for which the
claim is false. Write $k$ for the rank of the matrix $(a_{n,i}^{j})_{(i,j)\in I\times J}$.
Since the claim is false for $n$ we have $k<m-d$. Let $F\subset\{1,...,m\}$
be such that $|F|=m-d$, $|F\cap J|=k$ and $P_{n}^{F}(\lambda)\ne0$.

Let $1\le l\le d$ be with $l\in F$. For $j\in J\setminus F$ write
$F_{j}=(F\setminus\{l\})\cup\{j\}$. From $|F\cap J|=k$ it follows
that,
\begin{equation}
P_{n}^{F_{j}}(\lambda)=0\text{ for all }j\in J\setminus F\:.\label{eq:det of F_j=00003D0}
\end{equation}
Let $D$ be the determinant of the matrix obtained by replacing the
column vector $(a_{n,i}^{l})_{i=1}^{m-d}$ with the column vector
$(\rho_{n,i}^{l})_{i=1}^{m-d}$ in the matrix $(a_{n,i}^{j})_{(i,j)\in I\times F}$.
Since $|\rho_{n,i}^{l}|\le\delta^{n}$ for each $1\le i\le q_{n}$
and $a_{n,i}^{j}=O_{\lambda,m}(1)$ for each $(i,j)\in B_{n}$ it
follows that $D=O_{\lambda,m}(\delta^{n})$.

By (\ref{eq:basis at beginning}) it follows that $t_{j}^{l}=0$ for
each $j\in\{1,...,d\}\setminus\{l\}$. Thus for every $1\le i\le m-d$,
\[
\rho_{n,i}^{l}=L_{n,i}(t^{l})=\sum_{j=1}^{m}a_{n,i}^{j}t_{j}^{l}=\sum_{j\in F}a_{n,i}^{j}t_{j}^{l}+\sum_{j\in J\setminus F}a_{n,i}^{j}t_{j}^{l}\:.
\]
By solving these equations in $\{t_{j}^{l}\}_{j\in F}$ and applying
Cramer's rule we get,
\[
1=t_{l}^{l}=\frac{1}{P_{n}^{F}(\lambda)}\left(D-\sum_{j\in J\setminus F}t_{j}^{l}P_{n}^{F_{j}}(\lambda)\right)\:.
\]
From this, (\ref{eq:det of F_j=00003D0}) and $D=O_{\lambda,m}(\delta^{n})$
it follows that $P_{n}^{F}(\lambda)=O_{\lambda,m}(\delta^{n})$.

On the other hand, by (\ref{eq:control on product of poly}) and since
$P_{n,i}^{j}\in\mathcal{P}(1,n)$ for each $(i,j)\in B_{n}$,
\[
P_{n}^{F}\in\mathcal{P}((m!)n^{m(m+1)},mn)\:.
\]
From this, $P_{n}^{F}(\lambda)\ne0$ and Lemma \ref{lem:lb on val of poly},
\[
|P_{n}^{F}(\lambda)|\ge((m!)n^{m(m+1)})^{-M}\cdot M^{-mn}\:.
\]
Thus, if $\delta$ is taken to be small enough with respect to $\lambda$
and $m$ then $P_{n}^{F}(\lambda)=O_{\lambda,m}(\delta^{n})$ is not
possible. This contradiction completes the proof of the claim.
\end{proof}
For $j\in J$ and $1\le l\le d$ write $J_{j,l}=(J\setminus\{j\})\cup\{l\}$.
\begin{claim*}
Let $k\in J$ and $1\le l\le d$ be given, then $t_{k}^{l}=-P_{N}^{J_{k,l}}(\lambda)/P_{N}^{J}(\lambda)$.
\end{claim*}
\begin{proof}
Let $n\ge N$ and write $D$ for the determinant of the matrix obtained
by replacing the column vector $(a_{n,i}^{k})_{i=1}^{m-d}$ with the
column vector $(\rho_{n,i}^{l})_{i=1}^{m-d}$ in the matrix $(a_{n,i}^{j})_{(i,j)\in I\times J}$.
Since $|\rho_{n,i}^{l}|\le\delta^{n}$ for each $1\le i\le q_{n}$
and $a_{n,i}^{j}=O_{\lambda,m}(1)$ for each $(i,j)\in B_{n}$, it
follows that $D=O_{\lambda,m}(\delta^{n})$. By (\ref{eq:basis at beginning})
it follows that for every $1\le i\le m-d$,
\[
\rho_{n,i}^{l}=L_{n,i}(t^{l})=\sum_{j=1}^{m}a_{n,i}^{j}t_{j}^{l}=a_{n,i}^{l}+\sum_{j=d+1}^{m}a_{n,i}^{j}t_{j}^{l}\:.
\]
By solving these equations in $\{t_{j}^{l}\}_{j\in J}$ and applying
Cramer's rule, we get
\begin{equation}
t_{k}^{l}=\frac{1}{P_{n}^{J}(\lambda)}(D-P_{n}^{J_{k,l}}(\lambda))\;.\label{eq:by Cramer t^l_k =00003D}
\end{equation}

By (\ref{eq:control on product of poly}) and since $P_{n,i}^{j}\in\mathcal{P}(1,n)$
for each $(i,j)\in B_{n}$,
\[
P_{n}^{J}\in\mathcal{P}((m!)n^{m(m+1)},mn)\:.
\]
Thus, from $P_{n}^{J}(\lambda)\ne0$ and Lemma \ref{lem:lb on val of poly},
\[
|P_{n}^{J}(\lambda)|\ge((m!)n^{m(m+1)})^{-M}\cdot M^{-mn}\:.
\]
From this, $D=O_{\lambda,m}(\delta^{n})$ and by assuming that $\delta$
is small enough with respect to $\lambda$ and $m$, we get $|D/P_{n}^{J}(\lambda)|\le\delta^{n/2}$.
Hence from (\ref{eq:by Cramer t^l_k =00003D}) it follows that for
every $n\ge N$,
\[
\left|t_{k}^{l}+\frac{P_{n}^{J_{k,l}}(\lambda)}{P_{n}^{J}(\lambda)}\right|\le\delta^{n/2},
\]
and so,
\[
\left|\frac{P_{n}^{J_{k,l}}(\lambda)}{P_{n}^{J}(\lambda)}-\frac{P_{n+1}^{J_{k,l}}(\lambda)}{P_{n+1}^{J}(\lambda)}\right|\le2\delta^{n/2}\:.
\]
From this and $P_{n}^{J}(\lambda),P_{n+1}^{J}(\lambda)=O_{\lambda,m}(1)$,
we get
\[
\left|P_{n}^{J_{k,l}}(\lambda)P_{n+1}^{J}(\lambda)-P_{n+1}^{J_{k,l}}(\lambda)P_{n}^{J}(\lambda)\right|=O_{\lambda,m}(\delta^{n/2})\;.
\]

Now set,
\[
Q_{n}(X)=P_{n}^{J_{k,l}}(X)P_{n+1}^{J}(X)-P_{n+1}^{J_{k,l}}(X)P_{n}^{J}(X),
\]
then $Q_{n}(\lambda)=O_{\lambda,m}(\delta^{n/2})$. Moreover, by (\ref{eq:control on product of poly})
and since,
\[
P_{n}^{J},P_{n}^{J_{k,l}},P_{n+1}^{J},P_{n+1}^{J_{k,l}}\in\mathcal{P}((m!)(n+1)^{m(m+1)},m(n+1)),
\]
we have,
\[
Q_{n}\in\mathcal{P}(m^{5(m+1)}(n+1)^{4m(m+1)},2m(n+1))\:.
\]
Thus from Lemma \ref{lem:lb on val of poly} it follows that $Q_{n}(\lambda)=0$
or,
\[
|Q_{n}(\lambda)|\ge(m^{5(m+1)}(n+1)^{4m(m+1)})^{-M}\cdot M^{-2m(n+1)}\:.
\]
From this, $Q_{n}(\lambda)=O_{\lambda,m}(\delta^{n/2})$ and by assuming
that $\delta$ is small enough with respect to $\lambda$ and $m$,
it follows that we must have $Q_{n}(\lambda)=0$. We have thus shown
that for every $n\ge N$,
\[
\left|t_{k}^{l}+\frac{P_{n}^{J_{k,l}}(\lambda)}{P_{n}^{J}(\lambda)}\right|\le\delta^{n/2}\text{ and }\frac{P_{n}^{J_{k,l}}(\lambda)}{P_{n}^{J}(\lambda)}=\frac{P_{n+1}^{J_{k,l}}(\lambda)}{P_{n+1}^{J}(\lambda)},
\]
which clearly completes the proof of the claim.
\end{proof}
From the last claim and (\ref{eq:basis at beginning}) it follows
that for every $1\le l\le d$ and $1\le i\le m-d$,
\begin{multline*}
L_{N,i}(t^{l})=\sum_{j=1}^{m}a_{N,i}^{j}t_{j}^{l}=a_{N,i}^{l}-\sum_{j=d+1}^{m}a_{N,i}^{j}\frac{P_{N}^{J_{j,l}}(\lambda)}{P_{N}^{J}(\lambda)}\\
=a_{N,i}^{l}-\sum_{j=d+1}^{m}a_{N,i}^{j}\frac{\det\left((a_{N,i_{0}}^{j_{0}})_{(i_{0},j_{0})\in I\times J_{j,l}}\right)}{\det\left((a_{N,i_{0}}^{j_{0}})_{(i_{0},j_{0})\in I\times J}\right)}\:.
\end{multline*}
Thus by Cramer's rule we have $L_{N,i}(t^{l})=0$. Since $L_{N,1}\in\mathcal{L}_{N}$
there exist $w_{1},w_{2}\in\Lambda^{N}$ such that $w_{1}\ne w_{2}$,
$\lambda_{w_{1}}=\lambda_{w_{2}}$ and $L_{N,1}=L_{w_{1},w_{2}}$.
For each $1\le l\le d$,
\[
0=L_{w_{1},w_{2}}(t^{l})=\varphi_{t^{l},w_{1}}(0)-\varphi_{t^{l},w_{2}}(0),
\]
which together with $\lambda_{w_{1}}=\lambda_{w_{2}}$ implies $\varphi_{t,w_{1}}=\varphi_{t,w_{2}}$.
We have thus shown that when
\[
\underset{n}{\liminf}\:r_{n}=m-d
\]
the IFS $\Phi_{t}$ has exact overlaps. This completes the treatment
of the first case.

\subsection*{Second case}

Suppose next that for
\[
r=\underset{n}{\liminf}\:r_{n}
\]
we have $1\le r<m-d$ (recall that $r_{n}\ge1$ for all $n\ge N$).
Let $\{n_{k}\}_{k\ge1}$ be an increasing sequence of positive integers
with $n_{1}\ge N$ and $r_{n_{k}}=r$ for all $k\ge1$. Write $I=\{1,...,r\}$,
then without loss of generality we may assume that the row vectors
$\{a_{n_{k},i}\}_{i\in I}$ are linearly independent for all $k\ge1$.

Write $d'=m-r$ and $J=\{d'+1,...,m\}$, and note that $d<d'<m$.
Given $J_{0}\subset\{1,...,m\}$ with $|J_{0}|=r$ write $P_{k}^{J_{0}}(X)$
for the determinant of the matrix $(P_{n_{k},i}^{j}(X))_{(i,j)\in I\times J_{0}}$.
Since we no longer assume (\ref{eq:basis at beginning}), we may without
loss of generality assume that for infinitely many integers $k\ge1$,
\begin{equation}
|P_{k}^{J}(\lambda)|\ge|P_{k}^{J_{0}}(\lambda)|\text{ for all }J_{0}\subset\{1,...,m\}\text{ with }|J_{0}|=r\:.\label{eq:P^I is biggest}
\end{equation}
By moving to a subsequence without changing the notation we may suppose
that (\ref{eq:P^I is biggest}) holds for all $k\ge1$. Note that
since $\{a_{n_{k},i}\}_{i\in I}$ are independent this implies that
$P_{k}^{J}(\lambda)\ne0$.

For every $j\in J$ and $1\le l\le d'$ set $J_{j,l}=(J\setminus\{j\})\cup\{l\}$.
By moving to a subsequence without changing the notation we may assume
that there exist numbers $s_{0,j}^{l}\in[-1,1]$ such that,
\begin{equation}
s_{0,j}^{l}=-\underset{k}{\lim}\:\frac{P_{k}^{J_{j,l}}(\lambda)}{P_{k}^{J}(\lambda)}\text{ for all }j\in J\text{ and }1\le l\le d'\:.\label{eq:lim of det exist}
\end{equation}
For $1\le j\le d'$ and $1\le l\le d'$ set $s_{0,j}^{l}=\delta_{j,l}$,
where $\delta_{j,l}$ is the Kronecker delta. For each $1\le j\le m$
write $s_{0,j}=(s_{0,j}^{l})_{l=1}^{d'}$, and note that $(s_{0,j})_{j=1}^{d'}$
is the standard basis of $\mathbb{R}^{d'}$. Set $s_{0}=(s_{0,j})_{j=1}^{m}$,
so that $s_{0}\in\mathbb{R}^{d'm}$. For $1\le l\le d'$ write $s_{0}^{l}=(s_{0,j}^{l})_{j=1}^{m}$.

For $k\ge1$, $1\le j\le d'$ and $1\le l\le d'$ set $s_{k,j}^{l}=s_{0,j}^{l}$.
For $j\in J$ and $1\le l\le d'$ write,
\[
s_{k,j}^{l}=-P_{k}^{J_{j,l}}(\lambda)/P_{k}^{J}(\lambda)\:.
\]
For each $1\le j\le m$ set $s_{k,j}=(s_{k,j}^{l})_{l=1}^{d'}$ and
write $s_{k}=(s_{k,j})_{j=1}^{m}\in\mathbb{R}^{d'm}$. For $1\le l\le d'$
write $s_{k}^{l}=(s_{k,j}^{l})_{j=1}^{m}$. Since $(s_{k,j})_{j=1}^{d'}$
is the standard basis of $\mathbb{R}^{d'}$ for every $i\in I$ we
have,
\begin{multline*}
L_{n_{k},i}(s_{k}^{l})=\sum_{j=1}^{m}a_{n_{k},i}^{j}s_{k,j}^{l}=a_{n_{k},i}^{l}-\sum_{j=d'+1}^{m}a_{n_{k},i}^{j}\frac{P_{k}^{J_{j,l}}(\lambda)}{P_{k}^{J}(\lambda)}\\
=a_{n_{k},i}^{l}-\sum_{j=d'+1}^{m}a_{n_{k},i}^{j}\frac{\det\left((a_{n_{k},i_{0}}^{j_{0}})_{(i_{0},j_{0})\in I\times J_{j,l}}\right)}{\det\left((a_{n_{k},i_{0}}^{j_{0}})_{(i_{0},j_{0})\in I\times J}\right)}\:.
\end{multline*}
Thus by Cramer's rule,
\begin{equation}
L_{n_{k},i}(s_{k}^{l})=0\text{ for each }i\in I\text{ and }1\le l\le d'\:.\label{eq:functinals are 0}
\end{equation}
Since $\{a_{n_{k},i}\}_{i\in I}$ are independent and $r_{n_{k}}=r$,
\[
\{a_{n_{k},i}\}_{i=1}^{q_{n_{k}}}\subset\mathrm{span}\{a_{n_{k},i}\}_{i\in I}\:.
\]
Hence by (\ref{eq:functinals are 0}),
\begin{equation}
L_{n_{k},i}(s_{k}^{l})=0\text{ for all }1\le i\le q_{n_{k}}\text{ and }1\le l\le d'\:.\label{eq:functionals are 0 for all}
\end{equation}

Let $w_{1},w_{2}\in\Lambda^{n_{k}}$ be such that $w_{1}\ne w_{2}$
but $\varphi_{t,w_{1}}$ and $\varphi_{t,w_{2}}$ belong to the same
atom of the partition $\mathcal{E}_{\left\lceil -\log\delta\right\rceil n_{k}}\vee\mathcal{F}$.
By the the definition of $\mathcal{F}$ we have $\lambda_{w_{1}}=\lambda_{w_{2}}$
and by the definition of $\mathcal{E}_{\left\lceil -\log\delta\right\rceil n_{k}}$,
\[
|L_{w_{1},w_{2}}(t^{l})|=|\varphi_{t^{l},w_{1}}(0)-\varphi_{t^{l},w_{2}}(0)|\le\delta^{n_{k}}
\]
for each $1\le l\le d$. This shows that $L_{w_{1},w_{2}}\in\mathcal{A}_{n_{k}}$
and so by (\ref{eq:functionals are 0 for all}),
\[
\varphi_{s_{k}^{l},w_{1}}(0)-\varphi_{s_{k}^{l},w_{2}}(0)=L_{w_{1},w_{2}}(s_{k}^{l})=0
\]
for each $1\le l\le d'$. Since $\lambda_{w_{1}}=\lambda_{w_{2}}$
it follows that $\varphi_{s_{k},w_{1}}=\varphi_{s_{k},w_{2}}$.

We have thus shown that $\varphi_{s_{k},w_{1}}=\varphi_{s_{k},w_{2}}$
for every $w_{1},w_{2}\in\Lambda^{n_{k}}$ such that $\varphi_{t,w_{1}}$
and $\varphi_{t,w_{2}}$ belong to the same atom of $\mathcal{E}_{\left\lceil -\log\delta\right\rceil n_{k}}\vee\mathcal{F}$.
This clearly implies that,
\[
H(\nu_{s_{k}}^{(n_{k})})\le H(\nu_{t}^{(n_{k})},\mathcal{E}_{\left\lceil -\log\delta\right\rceil n_{k}}\vee\mathcal{F})\:.
\]
From this and (\ref{eq:small ent}),
\[
\frac{1}{\chi n_{k}}H(\nu_{s_{k}}^{(n_{k})})<\dim\mu_{t}+\epsilon<\beta-\epsilon\:.
\]
Hence $\dim\mu_{s_{k}}<\beta-\epsilon$ by Corollary \ref{cor:ub on dim}.

Note that by (\ref{eq:lim of det exist}) and the definitions of $s_{0}$
and $\{s_{k}\}_{k\ge1}$ it follows that $s_{k}\overset{k}{\rightarrow}s_{0}$.
Thus from Theorem \ref{thm:lower semi-cont of dim} we get that $\dim\mu_{s_{0}}\le\beta-\epsilon$.
Also, since $(s_{0,j})_{j=1}^{d'}$ is the standard basis of $\mathbb{R}^{d'}$,
\[
\mathrm{span}\{s_{0,1},...,s_{0,m}\}=\mathbb{R}^{d'}\:.
\]
Now because $d<d'<m$ we may use the induction hypothesis in order
to conclude that $\Phi_{s_{0}}$ has exact overlaps. It remains to
show that this implies that $\Phi_{t}$ also has exact overlaps.
\begin{claim*}
For every $1\le j\le m$ we have $t_{j}=\sum_{l=1}^{d'}s_{0,j}^{l}t_{l}$.
\end{claim*}
\begin{proof}
For every $1\le j\le d'$,
\[
\sum_{l=1}^{d'}s_{0,j}^{l}t_{l}=\sum_{l=1}^{d'}\delta_{j,l}t_{l}=t_{j}\:.
\]
Thus the claim holds for $1\le j\le d'$.

Let $j_{0}\in J$, $1\le l\le d$ and $k\ge1$ be given. Write $D_{k}$
for the determinant of the matrix obtained by replacing the column
vector $(a_{n_{k},i}^{j_{0}})_{i=1}^{r}$ with the column vector $(\rho_{n_{k},i}^{l})_{i=1}^{r}$
in the matrix $(a_{n_{k},i}^{j})_{(i,j)\in I\times J}$. Since $|\rho_{n_{k},i}^{l}|\le\delta^{n_{k}}$
for each $1\le i\le q_{n_{k}}$ and $a_{n_{k},i}^{j}=O_{\lambda,m}(1)$
for each $(i,j)\in B_{n_{k}}$, it follows that $D_{k}=O_{\lambda,m}(\delta^{n_{k}})$.
By (\ref{eq:control on product of poly}) and since $P_{n_{k},i}^{j}\in\mathcal{P}(1,n_{k})$
for each $(i,j)\in B_{n_{k}}$,
\[
P_{k}^{J}\in\mathcal{P}((m!)n_{k}^{m(m+1)},mn_{k})\:.
\]
From this, $P_{k}^{J}(\lambda)\ne0$ and Lemma \ref{lem:lb on val of poly},
\[
|P_{k}^{J}(\lambda)|\ge((m!)n_{k}^{m(m+1)})^{-M}\cdot M^{-mn_{k}}\:.
\]
Thus, by taking $\delta$ to be small enough with respect to $\lambda$
and $m$ we may assume that,
\[
|D_{k}/P_{k}^{J}(\lambda)|\le\delta^{n_{k}/2}\:.
\]

For every $1\le i\le r$,
\[
\rho_{n_{k},i}^{l}=L_{n_{k},i}(t^{l})=\sum_{j=1}^{m}a_{n_{k},i}^{j}t_{j}^{l}\:.
\]
Thus by solving these equations in $\{t_{j}^{l}\}_{j\in J}$ and applying
Cramer's rule,
\[
t_{j_{0}}^{l}=\frac{1}{P_{k}^{J}(\lambda)}(D_{k}-\sum_{j=1}^{d'}t_{j}^{l}P_{k}^{J_{j_{0},j}}(\lambda))\;.
\]
From this, $|D_{k}/P_{k}^{J}(\lambda)|\le\delta^{n_{k}/2}$ and (\ref{eq:lim of det exist}),
\[
t_{j_{0}}^{l}=-\underset{k}{\lim}\:\sum_{j=1}^{d'}t_{j}^{l}\frac{P_{k}^{J_{j_{0},j}}(\lambda)}{P_{k}^{J}(\lambda)}=\sum_{j=1}^{d'}s_{0,j_{0}}^{j}t_{j}^{l}\:.
\]
Since this holds for all $1\le l\le d$ it follows that $t_{j_{0}}=\sum_{j=1}^{d'}s_{0,j_{0}}^{j}t_{j}$,
which completes the proof of the claim.
\end{proof}
Now since $\Phi_{s_{0}}$ has exact overlaps there exist $n\ge1$
and $w_{1},w_{2}\in\Lambda^{n}$ such that $w_{1}\ne w_{2}$, $\lambda_{w_{1}}=\lambda_{w_{2}}$
and $L_{w_{1},w_{2}}(s_{0}^{l})=0$ for all $1\le l\le d'$. Write
$L=L_{w_{1},w_{2}}$ and let $c_{1},...,c_{m}\in\mathbb{R}$ be such
that,
\[
L(x)=\sum_{j=1}^{m}c_{j}x_{j}\text{ for }(x_{j})_{j=1}^{m}=x\in\mathbb{R}^{m}\:.
\]
Since $L(s_{0}^{l})=0$ for each $1\le l\le d'$ and by the last claim,
\[
0=\sum_{l=1}^{d'}L(s_{0}^{l})t_{l}=\sum_{j=1}^{m}c_{j}\sum_{l=1}^{d'}s_{0,j}^{l}t_{l}=\sum_{j=1}^{m}c_{j}t_{j}\:.
\]
This implies that $L(t^{l})=0$ for each $1\le l\le d$, which shows
that $\varphi_{t,w_{1}}=\varphi_{t,w_{2}}$.

We have thus shown that also in the second case $\Phi_{t}$ has exact
overlaps, which completes the proof of the theorem.
\end{proof}

$\newline$$\newline$\textsc{Centre for Mathematical Sciences,\newline Wilberforce Road, Cambridge CB3 0WA, UK}$\newline$$\newline$\textit{E-mail: }
\texttt{ariel.rapaport@mail.huji.ac.il}
\end{document}